\documentclass[10pt,reqno]{amsart}

\usepackage{amsmath,amssymb,amsthm,amsfonts}
\usepackage{latexsym,epsfig,enumerate,mathtools,xcolor,dsfont} 

\definecolor{cadmiumgreen}{rgb}{0.0, 0.42, 0.24}
\usepackage[
colorlinks, citecolor=cadmiumgreen,
pagebackref,
pdfauthor={David Harry Richman}, 
pdfstartview ={FitV},
]{hyperref}
\hypersetup{
	pdftitle={Lower rational approximations and Farey staircases},
}

\mathtoolsset{showonlyrefs}

\newtheorem{theorem}{Theorem}
\newtheorem{thm}[theorem]{Theorem}
\newtheorem{lemma}[theorem]{Lemma}

\newtheorem{prop}[theorem]{Proposition}
\newtheorem{corollary}[theorem]{Corollary}

\newtheorem{claim}[theorem]{Claim}

\theoremstyle{definition}
\newtheorem{dfn}[theorem]{Definition}
\newtheorem{rmk}[theorem]{Remark}
\newtheorem{eg}[theorem]{Example}

\newcommand{\floor}[1]{\left\lfloor{#1}\right\rfloor}
\newcommand{\ceil}[1]{\lceil{#1}\rceil}

\newcommand{\floorfrac}[2]{\floor{\frac{#1}{#2}}}
\newcommand{\ALR}{A}
\newcommand{\dilat}{\Delta}

\newcommand{\RR}{\mathbb{R}}
\newcommand{\ZZ}{\mathbb{Z}}
\newcommand{\NN}{\mathbb{N}}

\begin{document}

\title{Lower rational approximations and Farey staircases}

\author{David Harry Richman}
\address{Department of Mathematics, University of Washington, Seattle, WA 98195, USA}


\begin{abstract}
For a real number $x$, call $\frac1n \floor{nx}$ the 
$n$-th lower rational approximation of $x$.
We study the functions defined by
taking the cumulative average of the first $n$ lower rational approximations of $x$, which we call the Farey staircase functions.
This sequence of functions is monotonically increasing.
We determine limit behavior of these functions and show that they exhibit fractal structure under appropriate normalization.
\end{abstract}

\maketitle

\section{Introduction}
\label{sec:intro}

Let $\floor{x}$ denote the greatest integer no larger than $x$.
The function $x \mapsto \floor{x}$
is commonly called the {\em floor function}.
Its graph looks like a staircase consisting of unit-height steps at each integer;
we could also call $x \mapsto \floor{x}$ the {\em unit staircase function}.

For a positive integer $n$, the function $x \mapsto \frac1n \floor{nx}$ is a rescaled staircase function.
The rescaled staircase still has ``unit slope'' between steps, but each step now has size $1/n$.
The sequence of functions $\{ \frac1n \floor{nx} : n = 1,2,\ldots\}$ approaches the identity function $x$ from below,
but the convergence is not monotonic: it is not generally true that
$\frac1m \floor{mx} \leq \frac1n \floor{nx}$ if $m \leq n$.
\begin{figure}[h]
	\centering
	\makebox[0pt]{
	\includegraphics[width=1.4\textwidth]{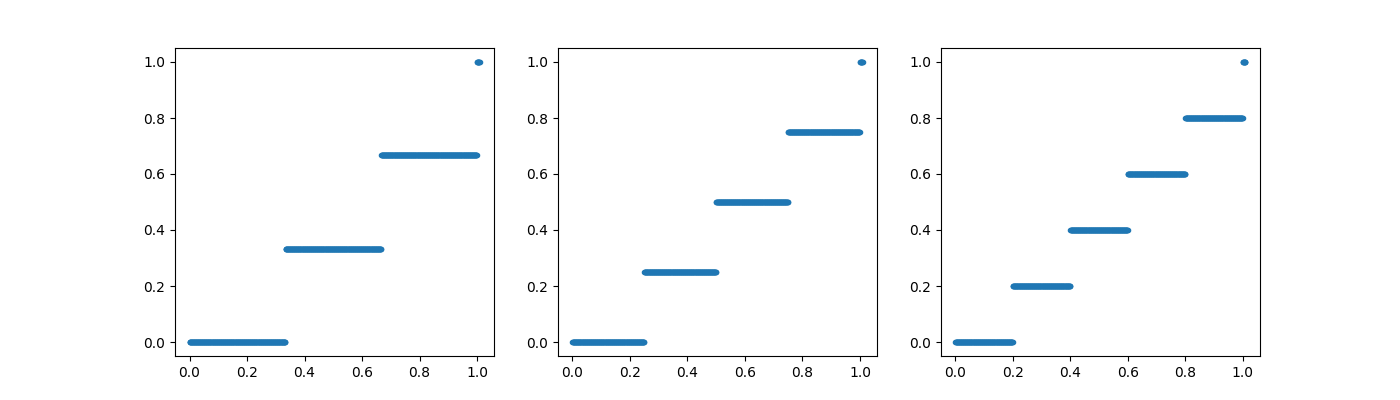}
	}
	\caption{Lower rational approximations $\frac1n \floor{nx}$, for $n = 3,4,5$.}
\end{figure}

Note that $\tfrac{1}{n} \floor{nx}$ is the largest element of $\frac{1}{n}\ZZ$ which is no larger than $x$; in symbols,
\begin{equation}
	\tfrac{1}{n} \floor{nx} = \max \{y \in \tfrac1n \ZZ : y \leq x \},
\end{equation}
which motivates us to call $\tfrac1n \floor{nx}$ the {\em $n$-th lower rational approximation} of $x$.

Consider taking the cumulative average of the first $n$ lower rational approximations, 
\begin{equation}
\label{eq:avg-lra}
	\ALR_n(x) := \frac1n \Big(\floor{x} + \tfrac12 \floor{2x} + \tfrac13 \floor{3x} +
	\cdots + \tfrac1n \floor{nx} \Big).
\end{equation}
Recall that the {\em Farey fractions} of order $n$ are all fractions whose denominator has size at most $n$;
\[
	\mathcal F_n = \left\{\frac{p}{q} : p \in \ZZ,\; q \in \NN,\; q \leq n \right\}.
\]
The function $\ALR_n(x)$ is a step function with jump discontinuities at the Farey fractions of order $n$.
We thus call the graph of $\ALR_n(x)$ the {\em Farey staircase} of order $n$.
See Figure~\ref{fig:farey-345} for small examples, and Figure~\ref{fig:avg-graph} for a larger one.
\begin{figure}[h]
	\centering
	\makebox[0pt]{
	\includegraphics[width=1.4\textwidth]{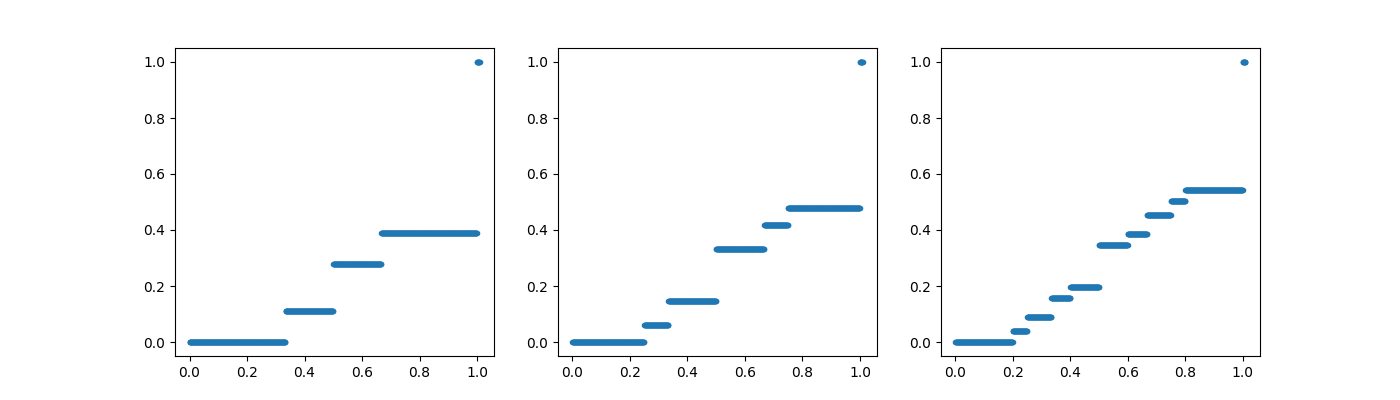}
	}
	\caption{Farey staircases $\ALR_n(x)$ for $n = 3, 4, 5$.}
	\label{fig:farey-345}
\end{figure}

A surprising property of the Farey staircase functions is that they are monotonically increasing, 
\[
	\ALR_1(x) \leq \ALR_2(x) \leq \ALR_3(x) \leq \cdots,
\]
their values approaching $x$ from below.
This is proved in \cite{klamkin,larson,richman};
see also Section~\ref{sec:intro-olympiad}.
In other words, if we define the incremented staircase function
\begin{equation}
	D_n(x) = \ALR_n(x) - \ALR_{n - 1}(x),
\end{equation}
then $D_n(x) \geq 0$ for all $x$ and all $n \geq 2$.
See Figure~\ref{fig:diff-graph-345} for small examples,
and Figures~\ref{fig:diff-graph} and \ref{fig:diff-graph-100} for larger ones.

\begin{figure}[h]
	\centering
	\makebox[0pt]{
	\includegraphics[width=1.4\textwidth]{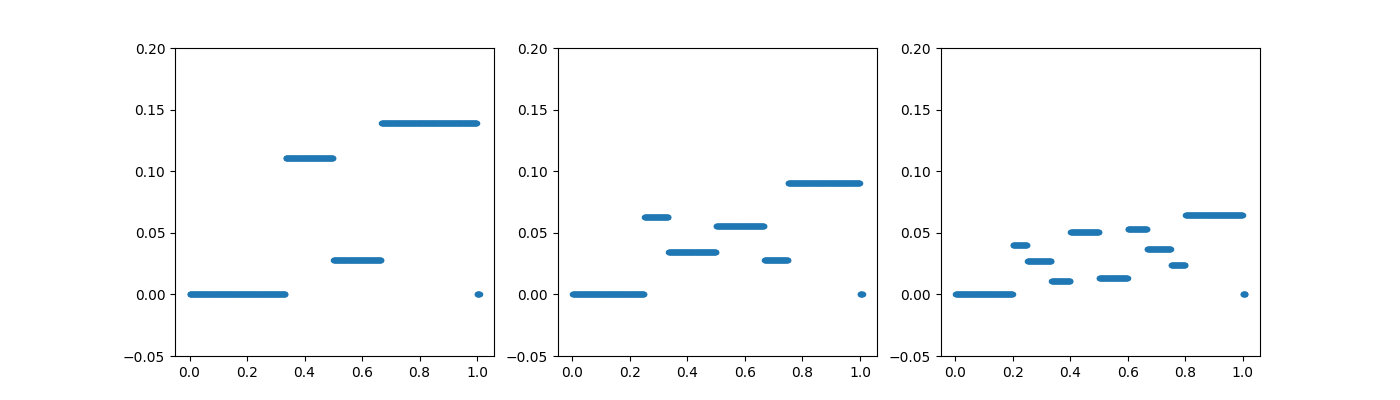}
	}
	\caption{Incremented staircases $D_n(x)$ for $n = 3, 4, 5$.}
	\label{fig:diff-graph-345}
\end{figure}

The purpose of this paper is to study the fractal behavior of the 
Farey staircases $\ALR_n(x)$ and the incremented staircases $D_n(x)$ 
that arises in the limit $n\to\infty$.
Hints of this fractal behavior are apparent in Figures~\ref{fig:avg-graph} and~\ref{fig:diff-graph}.
We will prove theorems which quantify aspects of this fractal-like behavior.

As a consequence, 
we obtain a novel perspective on why we should ``expect'' that $D_n(x)$ takes nonnegative values (under appropriate scaling) in the limit $n \to \infty$, 
which does not make claims about $D_n(x)$ for individual $n$.

\begin{figure}[h]
	\centering
	\includegraphics[scale=0.6]{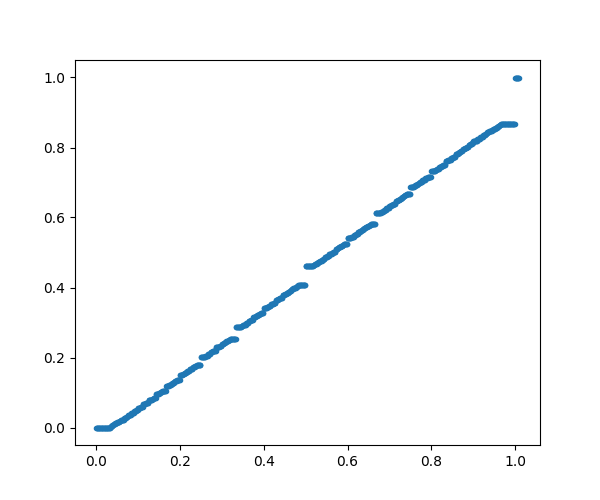}
	\caption{Farey staircase 
	$\ALR_{30}(x)$ on the domain $[0,1]$.}
	\label{fig:avg-graph}
\end{figure}

\subsection{Results}

To capture the fractal behavior of the Farey staircase,
we first study the following rescaling of $\ALR_n(x)$.
Let
\begin{equation} 
\label{eq:alr-zoom}
	B_n(x) := n \, \ALR_n(\tfrac1n x) 
	= \floor{\tfrac{1}{n}x} + \tfrac12 \floor{\tfrac{2}{n}x} + \tfrac13 \floor{\tfrac{3}{n}x} +
	\cdots + \tfrac1n \floor{x} .
\end{equation}
We obtain $B_n$ from $\ALR_n$ by ``zooming in'' at the origin by a factor of $n$.
In the limit $n\to \infty$, the sequence of functions $B_n$ converges to a pointwise limit.
\begin{thm}
\label{thm:limit-main}
Suppose $x \geq 0$.
As $n\to \infty$, 
\begin{equation}
\label{eq:B-limit}
	\lim_{n \to \infty} B_n(x) = \sum_{k = 1}^{\floor{x}} \log(x / k).
\end{equation}
\end{thm}
For a graph of the limit function, see the left side of Figure~\ref{fig:B-limit}.

\begin{rmk}
\label{rmk:main}
\hfill
\begin{enumerate}[(i)]
\item 
Let $B(x) \coloneqq \lim_{n \to \infty} B_n(x)$ denote
the limit function in \eqref{eq:B-limit}. 
Then $B(x)$ can be characterized as
\begin{equation}
	B(x)
	= \sum_{k = 1}^{\floor{x}} \left( \int_k^x \frac{1}{t} \,dt\right)
	= \int_0^x \frac{\floor{t}}{t} \,dt .
\end{equation}
Thus $B(x)$ is continuous on the domain $x \geq 0$, and we have the bound
\begin{equation}
\label{eq:B-vs-floor}
	B(x) 
	\;\leq\; \int_0^x \mathds{1}(t \geq 1) \,dt
	\;=\; x - 1 
	\;\leq\; \floor{x} .
\end{equation}
Note that, in particular, $B(x) = 0$ when $0 \leq x \leq 1$.

\item
If we let 
$\log^+(x) \coloneqq \max \{0, \log x\}$, then
\[
	B(x) 
	= \sum_{k = 1}^{\floor{x}} \log(x / k)
	= \sum_{k = 1}^\infty \log^+(x / k).
\]
Since $\log^+(x)$ is continuous on the domain $x \geq 0$, this also shows that $B(x)$ is continuous.

\item 
The limit function $B(x)$ 
can also be expressed as $\floor{x}  \log x - \log( \floor{x}! )$.
In other words,
\begin{equation*}
	\lim_{n\to\infty} B_n(x) = \begin{cases}
	0 &\text{if } 0 \leq x < 1, \\
	\log x &\text{if }1 \leq x < 2, \\
	2\log x - \log 2 &\text{if } 2 \leq x < 3 , \\
	\qquad \vdots \\
	k \log x - \log k! &\text{if }k \leq x < k+1 .
	\end{cases}
\end{equation*}
By using Stirling's approximation for $\log k!$, 
the bound~\eqref{eq:B-vs-floor} can be refined to the asymptotic 
\begin{equation*}
B(x) =  x - \frac12\log x + O(1) 
\qquad\text{as }x \to \infty.
\end{equation*}
\end{enumerate}
\end{rmk}

Theorem~\ref{thm:limit-main} captures the behavior of the Farey staircase $\ALR_n(x)$ in a small neighborhood of the origin.
The following theorem more generally describes the limiting behavior of $\ALR_n(x)$
in a small neighborhood above any rational point $x = \frac{p}{q}$.
\begin{theorem}
\label{thm:limit-extension}
Suppose $x \geq 0$.
For a positive reduced fraction $\frac{p}{q}$ (i.e. $p$ and $q$ are positive integers with $\gcd(p,q)=1$),
we have
\begin{equation*}
	\lim_{n \to \infty} \left(B_n(x + \tfrac{p}{q}n) - B_n(\tfrac{p}{q} n) \right) 
	= \frac{1}{q} \sum_{k = 1}^{\floor{qx}} \log(q x / k).
\end{equation*}
\end{theorem}
In other words, the limit function here is a rescaling of the limit function $B(x)$ in \eqref{eq:B-limit}, namely
\[
	\lim_{n \to \infty} \left(B_n(x + \tfrac{p}{q}n) - B_n(\tfrac{p}{q}n) \right) 
	= \frac1q B(qx) .
\]
The scaling factor $q$ is independent of the numerator $p$.
By Remark~\ref{rmk:main}, the limit function is equal to 
\[
	\frac1{q} B(qx) 
	\;=\; \displaystyle \frac1{q} \int_0^{qx} \frac{\floor{t}}{t} dt
	\;=\; \int_0^x \frac{\floor {qs}}{qs} ds
	\;=\; \frac1q \sum_{k = 1}^\infty \log^+\left( \frac{qx}{k} \right).
\]

\begin{figure}[h]
	\centering
	\includegraphics[scale=0.5]{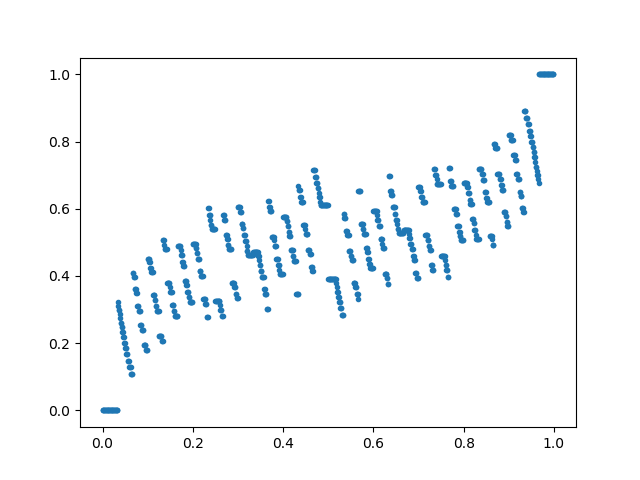}
	\caption{Incremented staircase 
	$D_{30}(x)$,
	normalized to height one.}
	\label{fig:diff-graph}
\end{figure}
Next we investigate fractal behavior that arises from taking the difference of consecutive Farey staircases.
As before let
\[
	D_n(x) \coloneqq \ALR_n(x) - \ALR_{n-1}(x)
\] 
for $n\geq 2$,
and let
$
	D_n^{\rm max} = \max\{ D_n(x) : x \in \RR\}.
$
Figure~\ref{fig:diff-graph} shows the graph of $D_{30}(x) / D_{30}^{\rm max}$.
Note that Figure~\ref{fig:diff-graph} shows a pattern of ``negatively-sloped stripes.''
The phenomenon of these segments persists in the limit $n \to \infty$, and is explained by the following result.
\begin{thm}
\label{thm:limit-diff}
Suppose $x \geq 0$. As $n\to \infty$,
\begin{equation}
	\lim_{n\to\infty} n^2 D_n\left(\tfrac1{n} x\right) 
	= \sum_{k = 1}^{\floor{x}} \left( 1 - \log(\tfrac{x}{k}) \right).
\end{equation}
\end{thm}
\begin{figure}[h]
	\centering
	\includegraphics[scale=0.5]{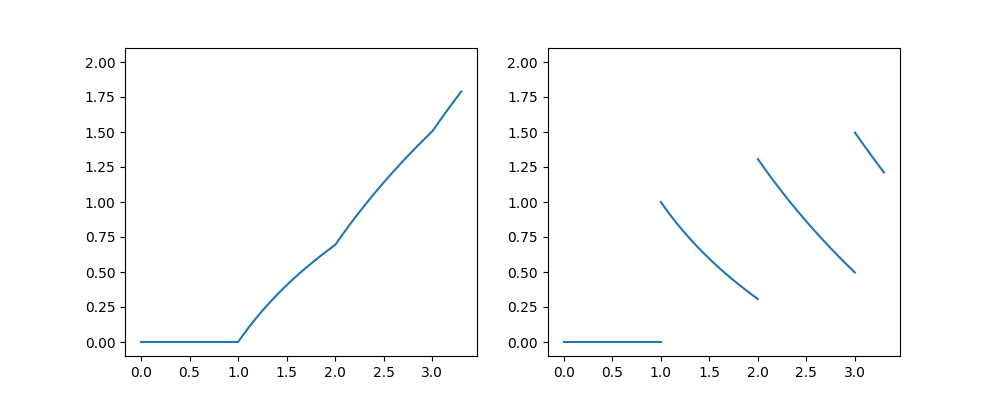}
	\caption{Limit functions $B(x) = \sum_{k = 1}^{\floor{x}} \log(x / k)$, left, and ${\floor{x} - B(x)}$, right.}
	\label{fig:B-limit}
\end{figure}

The limit function is in fact $\floor{x} - B(x)$,
where $B(x)$ is the limit function from Theorem~\ref{thm:limit-main}.
For a graph of the limit function, see the right side of Figure~\ref{fig:B-limit}.
In particular, the limit value of $n^2 D_n\left(\tfrac1{n} x\right)$ is nonnegative as a consequence of equation~\eqref{eq:B-vs-floor}.

By Stirling's approximation, 
we have
\begin{equation*}
	\lim_{n\to\infty} n^2 D_n\left(\frac1{n} x\right) 
	= \floor{x} - B(x) 
	=  \frac12\log x + O(1)
	\qquad \text{as } x \to \infty.
\end{equation*}


\begin{figure}[h]
\makebox[\textwidth]{
	\centering
	\includegraphics[scale=0.5]{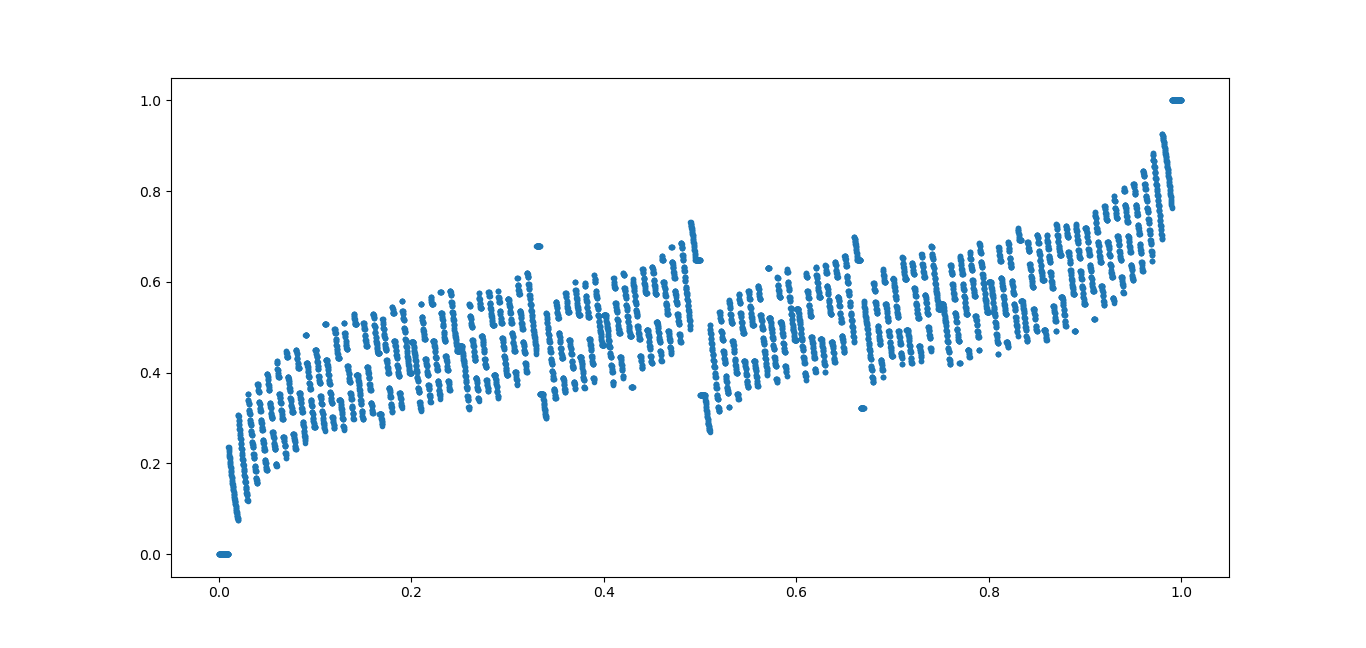}
}
	\caption{Incremented staircase 
$D_{100}(x)$, 
normalized to height one.}
	\label{fig:diff-graph-100}
\end{figure}

\subsection{Olympiad problem}
\label{sec:intro-olympiad}

Here we describe the initial motivation for this work.

Problem 5 of the 1981 U.S.A. Mathematical Olympiad was to prove that
\begin{equation}
\label{eq:olympiad}
	\floor{nx} \geq \floor{x} + \tfrac12 \floor{2x} + \tfrac13 \floor{3x} 
	+ \cdots + \tfrac1n \floor{nx}
\end{equation}
where 
$x$ is a real number, $n$ is a positive integer, and
$\floor{t}$ denotes the greatest integer less than or equal to $t$.  
Solutions can be found in \cite{klamkin,larson,richman}, and in Appendix~\ref{sec:olympiad} of this work.
%
With some algebraic manipulation,
the Olympiad inequality \eqref{eq:olympiad} is equivalent to the condition
$\ALR_{n - 1}(x) \leq \ALR_{n}(x)$
on Farey staircases.

Let
\[ 
	f_n(x) := 
	\floor{n x} - \left(\floor{x} + \tfrac12 \floor{2x} + \tfrac13 \floor{3x}
	+ \cdots + \tfrac1n\floor{n x}\right) ,
\]
and let $S$ denote the set values taken by $f_n$ for all $n$, 
\[
	S = \{ f_n(x) : x \in \RR,\, n = 1,2,\ldots \} .
\]
The Olympiad inequality states that $S$ does not contain negative values.
The following result of D. R. Richman~\cite{richman} gives a more complete description of $S$.
\begin{theorem}[see {\cite[Theorem 1.1]{richman}}]
Let $\lambda = 1 - \log 2$.
\begin{enumerate}[(i)]
\item 
The set $S$ is dense in the interval
$[ \lambda ,\, +\infty).$

\item 
The intersection $S \cap (-\infty, \lambda - \epsilon]$ has finitely many elements for any $\epsilon > 0$.
\end{enumerate}
\end{theorem}

Our motivation was to understand the appearance of the curious constant $1 - \log 2$ in the structure of $S$.
Note that
\[
	f_n\left( \frac1{n} x \right) = \floor{x} - B_n(x).
\]
The main result of this paper, Theorem~\ref{thm:limit-main}, 
characterizes the values of $f_n(x)$ in the limit $n \to \infty$, 
which implies the following corollary.
Let $S_n$ denote the set of values taken by $f_k$ for $k \geq n$, i.e.
\[
	S_n = \{ f_k(x) : x \in \RR,\, k \geq n \}.
\]

\begin{corollary}[to Theorem~\ref{thm:limit-main}]
For any $n \geq 1$,
the set $S_n$ is dense in $[1-\log 2, +\infty)$.
\end{corollary}

\subsection{Related work}

Interesting problems concerning combinations of floor functions were posed by Ramanujan~\cite{ramanujan}
and further generalized by Somu and Kukla~\cite{somu-kukla}.
Another useful floor function identity is attributed to Hermite,
with recent generalizations given by Aursukaree, Khemaratchatakumthorn, and Pongsriiam~\cite{aursukaree-et-al}.
Other sums of scaled floor function are considered by Thanatipanonda and Wong~\cite{thanatipanonda}.

Niederreiter~\cite{niederreiter} and
Dress~\cite{dress} proved bounds on the discrepancy of the Farey sequence, which concerns the spacing between consecutive fractions.
Kanemitsu and Yoshimoto~\cite{kanemitsu-yoshimoto} connect the Riemann hypothesis to certain estimates on sums of Farey fractions.
Lagarias and Mehta~\cite{lagarias-mehta} relate the Riemann hypothesis to properties of the product of Farey fractions of a given order.
Kunik~\cite{kunik} also studies a family of real-valued functions derived from Farey fractions, and their limiting behavior.

\subsection{Organization}
This paper is organized as follows. 
In Section~\ref{sec:preliminary} we derive some basic identities satisfied by the Farey staircase functions $\ALR_n(x)$ and their incremented functions $D_n(x)$.
In Section~\ref{sec:proofs} we prove Theorems~\ref{thm:limit-main}, \ref{thm:limit-extension}, and \ref{thm:limit-diff}.
In Section~\ref{sec:dilation} we give a heuristic argument for  Theorem~\ref{thm:limit-main} that provides an alternative perspective.
Finally in Appendix~\ref{sec:olympiad} we include a proof of the Olympiad inequality~\eqref{eq:olympiad}.

\section{Preliminaries}
\label{sec:preliminary}

Let $\mathcal F_n $ denote the Farey fractions of order $n$,
\[
	\mathcal F_n \coloneqq \bigcup_{m=1}^n \frac1{m} \ZZ,
\]
and let $\mathcal F_n^{[0,1]} = \mathcal F_n \cap [0,1]$
denote the Farey fractions on the unit interval.
Recall that the {\em Farey staircase function} of order $n$ is 
\[
	A_n(x) = \frac1n \sum_{k=1}^n \frac1k \floor{kx}
\]
and the {\em incremented staircase function} of order $n$ is 
\[
	D_n(x) = A_n(x) - A_{n-1}(x).
\]

The following identities are straightforward to verify.

\begin{prop}
\label{prop:diff-farey}
The functions $A_n(x)$ and $D_n(x)$ satisfy the following, for $n \geq 2$.
\hfill
\begin{enumerate}[(a)]
\item 
$\displaystyle
D_n(x) = \frac1{n} \left( \frac1n \floor{nx} - A_{n-1}(x)\right); 
$

\item 
$\displaystyle
D_n(x) = \frac1{n-1} \left( \frac1n \floor{nx} - A_n(x)\right).
$

\end{enumerate}
\end{prop}
\begin{proof}
Note that
\begin{equation}
\label{eq:alr-increment}
	n A_n(x) = \sum_{k = 1}^n \frac1{k} \floor{kx} = \frac1n \floor{nx} + (n-1)A_{n-1}(x).
\end{equation}
Solve the above equation for $A_n(x)$ and subtract $A_{n-1}(x)$, to obtain for $n \geq 2$
\begin{align*}
	A_n(x) - A_{n-1}(x)
	&= \frac1n \left( \frac1n \floor{nx} + (n-1)A_{n-1}(x) \right) - A_{n-1}(x) \\
	&= \frac1{n} \left( \frac1{n} \floor{nx} - A_{n-1}(x) \right), \\
\end{align*}
as claimed in (a).
If we instead use \eqref{eq:alr-increment} to solve for $A_{n-1}(x)$, then subtract from $A_n(x)$,
we obtain for $n \geq 2$ that
\begin{align*}
	A_n(x) - A_{n-1}(x) &= A_n(x) - \frac1{n-1} \left( - \frac1{n}\floor{nx} + n A_n(x) \right) \\
&= \frac1{n-1} \left( \frac1{n} \floor{nx} - A_n(x) \right),
\end{align*}
as claimed in (b).
\end{proof}

%
%

\begin{lemma}
\label{lem:staircase-bounds}
	\hfill
	\begin{enumerate}[(a)]
		\item For $n \geq 1$, we have 
		$\displaystyle A_n(x) \leq \frac{1}{n}\floor{nx} \leq A_n(x) + \frac{\log n}{n}$.
		\item
		For $n \geq 2$, we have
		$\displaystyle 0 \leq D_n(x) \leq \frac{\log n}{n(n - 1)}$.
	\end{enumerate}
\end{lemma}
\begin{proof}
	(a) The lower bound is equivalent to the USAMO inequality \eqref{eq:olympiad}; a proof can be found in \cite{klamkin}, \cite{larson}, or \cite[Lemma 2.1]{richman}.
	The upper bound follows from \cite[Theorem 6.1]{richman}, along with the harmonic sum bound 
	$\displaystyle \sum_{k = 2}^n \frac{1}{k} \leq \log n $.

	(b) Combine Proposition~\ref{prop:diff-farey} (b) with part (a) of this lemma.
\end{proof}

The following lemma allows us to interchange a sum for an integral.
\begin{lemma}
\label{lem:sum-integral}
Suppose $\{a_n\}$ and $\{ b_n \}$ are positive integer sequences which 
satisfy
\begin{itemize}
\item 
$a_n \leq b_n$, and

\item
$a_n \to \infty$ and $b_n \to \infty$
as $n\to \infty$.
\end{itemize}
Then 
$\displaystyle
\lim_{n\to \infty} \sum_{k = a_n }^{b_n} \frac1k = \lim_{n\to\infty} \int_{a_n}^{b_n} \frac{dt}{t} 
.$
\end{lemma}
\begin{proof}
By consideration of Riemann sums,
\[
	\sum_{k=a_n+1}^{b_n} \frac1k  
	\quad\leq\quad \int_{a_n}^{b_n} \frac{dt}{t}
	\quad\leq\quad \sum_{k=a_n}^{b_n-1} \frac1k ,
\]
which implies that
\[
	\frac1{a_n}
	\quad\geq\quad \left( \sum_{k=a_n}^{b_n} \frac1k \right) - \left( \int_{a_n}^{b_n} \frac{dt}{t} \right)
	\quad\geq\quad \frac1{b_n}.
\]
The result follows by taking $n \to \infty$,
since the hypotheses on $a_n$ and $b_n$ imply that
$\lim_{n\to\infty} 1/{a_n} = 0$ and $\lim_{n\to\infty} 1/{b_n} = 0$.
\end{proof}

\section{Proofs}
\label{sec:proofs}

\subsection{Farey staircase at $0$}
We are now ready to prove Theorem~\ref{thm:limit-main}, which states that
\begin{equation}
\label{eq:limit-repeat}
	\lim_{n \to \infty} B_n(x) 
	= \sum_{k = 1}^{\floor{x}} \log(x / k),
\end{equation}
if $x \geq 0$ and $B_n(x)$ denotes the function
\[
	B_n(x) \coloneqq \sum_{k = 1}^n \tfrac{1}{k} \floor{\tfrac{k}{n} x}
	= \floor{\tfrac{1}{n}x} 
	+ \tfrac12 \floor{\tfrac{2}{n}x} 
	+ \tfrac13 \floor{\tfrac{3}{n}x} +
	\cdots + \tfrac1n \floor{x} .
\]

\begin{proof}[Proof of Theorem~\ref{thm:limit-main}]
For $x$ in the range $0 \leq x < 1$,
each summand in $B_n(x)$ vanishes so $B_n(x) = 0$.

Suppose we fix $x$ in the range $1 \leq x < 2$. 
For $k \in \{1,2,\ldots, n\}$, we have
\[ 
	\floor{\tfrac{k}{n} x} = \begin{cases}
	0 & \text{if } 1\leq k < \tfrac{n}{x}, \\
	1 & \text{if } \tfrac{n}{x} \leq k \leq n .
	\end{cases}
\]
Then
\begin{align*}
	B_n(x) = \sum_{k=1}^n \tfrac{1}{k} \floor{\tfrac{k}{n} x}
	= \sum_{k=\ceil{n/x}}^n \tfrac{1}{k} .
\end{align*}
By Lemma~\ref{lem:sum-integral},
\[ 
	\lim_{n\to \infty} \sum_{k=\ceil{n/x}}^n \frac{1}{k}
	= \lim_{n\to\infty} \int_{\ceil{n/x}}^n \frac{dt}{t}
	= \lim_{n\to \infty} \left(\int_{n/x}^n \frac{dt}{t} - \int_{n/x}^{\ceil{n/x}} \frac{dt}{t} \right)
	= \log x 
\] 
since the integral $\int_{n/x}^{\ceil{n/x}} \frac{dt}{t} \leq \frac{x}{n} $ vanishes in the limit.

In general, suppose $x$ is in the range $M \leq x < M + 1$ for a positive integer $M$.
Then if we group the summands in $B_n(x)= \sum_{k=1}^n \tfrac{1}{k} \floor{\tfrac{k}{n} x}$ according to the value of $\floor{\tfrac{k}{n} x}$, we obtain
\begin{align}
	B_n(x) &= \sum_{k=\ceil{n/x}}^{\ceil{2n/x} - 1} \tfrac{1}{k} + \sum_{k = \ceil{2n/x}}^{\ceil{3n/x} - 1} \tfrac{2}{k} + \cdots + \sum_{k = \ceil{(M-1)n/x}}^{\ceil{Mn/x} - 1} \tfrac{M - 1}{k} + \sum_{k = \ceil{Mn/x}}^{n} \tfrac{M}{k} \\
	&= \sum_{k=\ceil{n/x}}^n \tfrac1k + \sum_{k=\ceil{2n/x}}^n \tfrac1{k}  + \cdots + \sum_{k=\ceil{Mn/x}}^{n} \tfrac1k .
\end{align}
Then we apply Lemma~\ref{lem:sum-integral} to each sum,
\begin{align*}
	\lim_{n\to\infty} B_n(x) &= \sum_{j=1}^M \lim_{n\to\infty} \sum_{k = \ceil{jn/x}}^n \frac1k 
	= \sum_{j=1}^M \lim_{n\to\infty} \int_{\ceil{jn/x}}^n \frac{dt}{t} \\
	&= \sum_{j=1}^M \lim_{n\to\infty} \left( \int_{jn/x}^n \frac{dt}{t} - \int_{jn/x}^{\ceil{jn/x}} \frac{dt}{t} \right) 
	= \sum_{j=1}^M  \log (x/j) .
\end{align*}
This completes the proof, since $M$ was chosen to be $M = \floor{x}$.
\end{proof}

\subsection{Farey staircase at $p/q$}

We next address the fractal behavior of $A_n(x)$.
For a reduced fraction $\frac{p}{q}$, we call
\[
	A_n\left(x + \frac{p}{q} \right) - A_n\left(\frac{p}{q}\right)
\]
the $(p/q)$-centered Farey staircase of order $n$.
After zooming in by a factor of $n$, we will show these functions converge to a limit as $n \to \infty$.
Recall from the introduction that 
$\displaystyle
	A_n(x) = \frac1{n} \sum_{k=1}^n \tfrac1k \floor{k x}
$
and
$\displaystyle
	B_n(x) = \sum_{k=1}^n \tfrac1k \floor{\tfrac{k}{n}x}.
$

Theorem~\ref{thm:limit-extension} states that
for a reduced fraction $\frac{p}{q}$ and $x \geq 0$,
\begin{equation*}
	\lim_{n \to \infty} \left(B_n\left(x + \frac{p}{q}n\right) - B_n\left( \frac{p}{q}n \right) \right) 
	= \frac1q \sum_{k = 1}^{\floor{qx}} \log\left( \frac{qx}{k} \right).
\end{equation*}
\begin{proof}[Proof of Theorem~\ref{thm:limit-extension}]
The summand in the expression $B_n(x + \frac{p}{q}n) - B_n(\frac{p}{q}n)$ is $\frac1k \floor{\frac{k}{n} x + \frac{kp}{q}} - \frac1k \floorfrac{kp}{q}$.
If $r \equiv kp \mod q$, then
\[
	\floor{ \frac{k}{n} x + \frac{kp}{q}} -  \floorfrac{kp}{q}
	= \floor{ \frac{k}{n} x + \frac{r}{q}} -  \floorfrac{r}{q}.
\]
Moreover, if the integer $r$ is chosen in the range $\{0, 1, 2, \ldots, q-1\}$, then $\floorfrac{r}{q} = 0$.
By grouping terms by the $q$-residue class of the product $kp$, we obtain
\begin{align}
	B_n(x + \tfrac{p}{q}n) - B_n(\tfrac{p}{q}n) 
	&= \sum_{k=1}^n \frac1k \left( \floor{ \frac{k}{n} x + \frac{kp}{q}} -  \floorfrac{kp}{q}\right) \\
	&= \sum_{r=0}^{q-1} \sum_{\substack{k=1 \\ kp \equiv r \, (q)}}^n \frac1k \floor{\frac{k}{n} x + \frac{r}{q}}. 
	\label{eq:with-r}
\end{align}
For the inner summation, we have
\begin{align}
	\sum_{\substack{k = 1 \\ kp \equiv r \, (q)}}^n \frac1k \floor{\frac{k}{n} x + \frac{r}{q}}
	&= \sum_{\substack{k = 1 \\ kp \equiv r \, (q)}}^n \frac1k \,\sum_{j = 1}^{\infty} \mathds{1}\left( j \leq \floor{\frac{k}{n} x + \frac{r}{q}} \right) \\
	&= \sum_{j = 1}^{\infty} \sum_{\substack{k = 1 \\ kp \equiv r \, (q)}}^n \frac1k \,\mathds{1}\left( j \leq \floor{\frac{k}{n} x + \frac{r}{q}} \right)
	\label{eq:4}
\end{align}
where $\mathds{1}$ denotes the indicator function.
Since the index $k$ ranges from $1$ to $n$, the maximal value of $j$ satisfying $j \leq \floor{\frac{k}{n}x + \frac{r}{q}}$ is equal to $j_\mathrm{max} = \floor{x + \frac{r}{q}} $.

The condition $j \leq \floor{\frac{k}{n} x + \frac{r}{q}}$ can be solved for $k$ as follows,
\[
	j \leq \floor{\frac{k}{n}x + \frac{r}{q}} 
	\quad\Leftrightarrow\quad j \leq \frac{k}{n}x + \frac{r}{q}
	\quad\Leftrightarrow\quad \frac{n}{x}\left(j - \frac{r}{q}\right) \leq k .
\]
Thus we can express the sum \eqref{eq:4} as
\begin{align}
	\sum_{\substack{k = 1 \\ kp \equiv r \, (q)}}^n \frac1k \floor{\frac{k}{n} x + \frac{r}{q}}
	&= \sum_{j = 1}^{j_\mathrm{max}} \sum_{\substack{k = 1 \\ kp \equiv r \, (q)}}^n \frac1k \,\mathds{1}\left( k \geq \frac{n}{x}\left(j - \frac{r}{q}\right) \right) \\
	&= \sum_{j = 1}^{\floor{x + \frac{r}{q}}}  \sum_{\substack{k = \ceil{\frac{n}{x}(j - \frac{r}{q})} \\ kp \equiv r \, (q)}}^n \frac{1}{k} .
\label{eq:7}
\end{align}
Since the function $1/t$ is sufficiently smooth, the inner summation of \eqref{eq:7} restricted to a single $q$-residue class converges to a $1/q$-factor of the unrestricted sum,
\[
	\lim_{n \to \infty} \sum_{\substack{k = \ceil{\frac{n}{x}(j - \frac{r}{q})} \\ kp \equiv r \, (q)}}^n \frac{1}{k} 
	= \frac1{q} \left( \lim_{n \to \infty} \sum_{k = \ceil{\frac{n}{x}(j - \frac{r}{q})}}^n \frac{1}{k} \right)
	= \frac1{q} \log\left( \frac{q x}{q j - r} \right).
\]
Summing these limits over the index $j$ as in \eqref{eq:7},
\begin{align}
	\lim_{n \to \infty} \sum_{\substack{k = 1 \\ kp \equiv r \, (q)}}^n \frac1k \floor{\frac{k}{n} x + \frac{r}{q}}
	&\;=\; \sum_{j = 1}^{\floor{x + \frac{r}{q}}} \lim_{n \to \infty} \sum_{\substack{k = \frac{n}{x}(j - \frac{r}{q}) \\ kp \equiv r \, (q)}}^n \frac{1}{k} 
	\;=\; \frac{1}{q}\sum_{j = 1}^{\floor{x + \frac{r}{q}}} \log\left( \frac{q x}{q j - r} \right).
\end{align}
Finally, we sum over the $q$-residue class representatives $r$ as in \eqref{eq:with-r},
\begin{align}
	\lim_{n \to \infty} \left( B_n\left(x + \frac{p}{q}n\right) - B_n\left(\frac{p}{q}n\right)\right) 
	&=\; \sum_{r = 0}^{q - 1} \frac{1}{q}\sum_{j = 1}^{\floor{x + \frac{r}{q}}} \log\left( \frac{q x}{q j - r} \right) \\
	&=\; \frac{1}{q} \sum_{r = 0}^{q - 1} \sum_{\substack{\ell = 1 \\ \ell \equiv -r (q)}}^{\floor{q x}} \log\left( \frac{q x}{\ell} \right) \\
	&=\; \frac{1}{q} \sum_{\ell = 1}^{\floor{qx}} \log\left( \frac{qx}{\ell} \right)
\end{align}
as claimed.
\end{proof}

\subsection{Incremented staircase at $0$}

Concerning the limiting behavior of 
$D_n(x) = A_n(x) - A_{n-1}(x)$, 
Theorem~\ref{thm:limit-diff}
states that
\[
	\lim_{n \to \infty} n^2 D_n\left( \frac1{n} x \right) = \sum_{k = 1}^{\floor{x}} \left( 1 - \log\left( \frac{x}{k} \right)\right).
\]
To address this limit, we combine Theorem~\ref{thm:limit-main} with Proposition~\ref{prop:diff-farey}.

\begin{proof}[Proof of Theorem~\ref{thm:limit-diff}]
From Proposition~\ref{prop:diff-farey}, we have
\[
	(n^2 - n) D_n\left(\frac1{n} x\right) = \floor{x} - n A_{n}\left(\frac1{n} x\right) 
	= \floor{x} - B_n(x).
\]
In the limit $n \to \infty$, the right-hand side is equal to 
\[
	\floor{x} - B(x) = \sum_{k = 1}^{\floor{x}} \left( 1 - \log\left( \frac{x}{k} \right)\right).
\]
To complete the proof, it suffices to verify the uniform convergence $ n D_n(x) \to 0$ as $n \to \infty$.
This follows from Lemma~\ref{lem:staircase-bounds} (b), which states
\[
	0 \leq n D_n(x) \leq \frac{\log n}{n - 1}
	\qquad\text{for all } n \geq 2. \qedhere
\]
\end{proof}

\section{The dilation derivative}
\label{sec:dilation}

In this section, we introduce some additional conceptual framework that suggests a heuristic argument for Theorem~\ref{thm:limit-main}.
We do not turn this heuristic into a rigorous alternative proof.

\begin{dfn}
Given a function $f: \RR \to \RR$,
let
\begin{equation}
\label{eq:dilat-deriv}
	\dilat f(x) = \lim_{\lambda \to 1} \frac{\lambda^{-1} f(\lambda x) - f(x) }{\lambda-1 } ,
\end{equation}
when the limit exists.
We call $\dilat f$ the {\em dilation derivative} of $f$.
\end{dfn}
Compared to the limit definition of the usual derivative, the horizontal translation $f(x + \epsilon)$ is replaced with the dilation 
$\lambda^{-1} f(\lambda x)$,
with $\lambda = 1 + \epsilon$.

\begin{eg}
The dilation derivative satisfies the following identities.
\begin{enumerate}
\item 
For any function $f$, we have $\dilat f(0) = -f(0)$;

\item 
$\dilat x^{\alpha} = (\alpha - 1) x^{\alpha}$;

\item 
$\dilat e^{kx} = (k x - 1) e^{kx}$;

\item 
$\dilat \log x = 1 - \log x$.
\end{enumerate}
\end{eg}
These examples generalize to the following result.

\begin{prop}
For a differentiable function $f(x)$,
\[
	\dilat f(x) = x \frac{d}{dx} f(x) - f(x) .
\]
\end{prop}
\begin{proof}
If $x = 0$, the claim follows from the definition of $\dilat f(x)$. 
Now suppose $x \neq 0$.
We have
\begin{align}
	\dilat f(x) = \lim_{\lambda \to 1} \frac{1}{\lambda} \cdot \frac{f(\lambda x) - \lambda f(x)}{\lambda - 1}
	&= \lim_{\lambda \to 1} \frac{f(\lambda x) - f(x) + f(x) - \lambda f(x)}{\lambda - 1}  \\
	&= \left( \lim_{\lambda \to 1} \frac{f(\lambda x) - f(x)}{\lambda - 1} \right) - f(x) \\
	&= x \left( \lim_{\lambda \to 1} \frac{f(x + (\lambda - 1)x) - f(x)}{(\lambda - 1)x}\right) - f(x) .
\end{align}
The limit in the last expression is $\frac{d}{dx} f(x)$, since $\epsilon \coloneqq (\lambda - 1) x \to 0$ as $\lambda \to 1$.
\end{proof}

\begin{claim}[Dilation Heuristic]
\label{claim:heuristic}
Suppose that $B(x)$ is the pointwise limit of $B_n(x)$.
Then
\[
	\dilat B(x) = \floor{x} - B(x).
\]
\end{claim}
Proposition~\ref{prop:diff-farey} implies that
\[
	(n - 1) A_n (x) - (n - 1) A_{n-1}( x) = \frac{1}{n}\left( \floor{n x} - n A_n(x) \right).
\]
Replace $x$ with $\frac1{n - 1} x$, to yield
\[
	\frac{n-1}{n} B_n\left(\frac{n}{n-1} x\right) - B_{n-1}\left(x \right) = \frac1{n} \left( \floor{\frac{n}{n-1} x} - B_n\left(\frac{n}{n-1} x\right) \right).
\]
Now let $\lambda_n = \frac{n}{n-1}$,
so the equation above multiplied by $n-1$ becomes
\begin{equation}
\label{eq:heuristic-limit}
	\frac{\lambda_n^{-1} B_n(\lambda_n x) - B_{n-1}(x)}{\lambda_n-1}
	= \lambda_n^{-1} \left( \floor{\lambda_n x} - B_n\left(\lambda_n x\right) \right).
\end{equation}
Then take the limit as $n\to \infty$, so that $\lambda_n \to 1$.
The right-hand side approaches $\floor{x} - B(x)$.
The left-hand side is more subtle:
by hypothesis $B_n \to B$, so we may compare the left-hand side with
\begin{equation}
	\lim_{\lambda_n \to 1} \frac{\lambda_n^{-1} B(\lambda_n x) - B(x)}{\lambda_n - 1} = \dilat B(x).
\end{equation}
We would like to say that as $n \to \infty$, the left-hand side of \eqref{eq:heuristic-limit} approaches $\dilat B(x)$.
However, the key step that remains in order to 
reach this conclusion rigorously is to bound the difference
\[
	\left| \frac{\lambda_n^{-1} B_n(\lambda_n x) - B_{n-1}(x)}{\lambda_n - 1} - 	\frac{\lambda_n^{-1} B(\lambda_n x) - B(x)}{\lambda_n - 1} \right|.
\]
To do so, it is necessary to further analyze the rate of convergence $B_n(x) \to B(x)$,
which we leave for future investigation.

\begin{corollary}[to Claim~\ref{claim:heuristic}]
The limit function $B(x)$ is a solution to the differential equation
\[
	x B'(x) = \floor{x}, \qquad B(0) = 0 .
\]
\end{corollary}

If we know additionally that the limit function $B(x)$ is continuous, this implies that $\displaystyle B(x) = \int_0^x \frac{\floor{t}}{t} dt$, recovering Theorem~\ref{thm:limit-main} via Remark~\ref{rmk:main} (ii).

\appendix 
\section{Proof of Olympiad problem}
\label{sec:olympiad}

We conclude with a proof of the Olympiad problem \eqref{eq:olympiad} for the sake of completeness, which we restate below.
\begin{claim}[1981 USAMO, Problem 5]
For any positive integer $n$ and any $x$, 
\begin{equation*}
	\floor{nx} \geq \sum_{k=1}^n \frac{1}{k}\floor{kx} .
\end{equation*}
\end{claim}
\begin{proof}
We proceed by induction on $n$.
The relation clearly holds when $n = 1$ 
for all $x$, so suppose $n\geq 2$.

Given $n$ and $x$,
let $d = d(n,x)$ denote any element of $\{1, 2, \dots, n\}$ such that 
\[
	\tfrac1d \floor{dx} = \max \left\{ \tfrac1k \floor{kx} : k = 1,2,\ldots, n\right\} .
\]
Note that $\floor{x + y} \geq \floor{x} + \floor{y}$ for any real numbers $x$ and $y$, so we have
\begin{align}
\label{eq:olymp-split}
	\floor{nx} 
	\geq \floor{(n-d)x} + \floor{dx},
\end{align}
and we can bound each of the right-hand terms separately.

The assumption that 
$\frac1d \floor{dx} \geq \frac{1}{k} \floor{kx}$
for $k = 1,2,\ldots , n$
implies that
\begin{equation}
\label{eq:olymp3}
	\floor{dx} = \sum_{k=n-d+1}^n \frac1d \floor{dx} \geq \sum_{k=n-d+1}^n \frac1k \floor{kx},
\end{equation}
and by the induction hypothesis, 
\begin{equation}
\label{eq:olymp4}
	\floor{(n-d)x} \geq \sum_{k=1}^{n-d} \frac1k \floor{kx}.
\end{equation}
Combining \eqref{eq:olymp-split} with the bounds \eqref{eq:olymp3} and \eqref{eq:olymp4} implies that
\[
	\floor{nx} \geq 
	\left(\sum_{k=1}^{n-d} \tfrac1k \floor{kx} \right) + \left(\sum_{k=n-d+1}^n \tfrac1k \floor{kx} \right)
\]
as desired.
\end{proof}

\section*{Acknowledgements}
I am grateful to Michael Filaseta for bringing this problem to my attention,
and to Jeffrey Lagarias and Samantha Pinella for useful conversations.
I thank the anonymous referee for suggestions that improved this work, in particular in the proof of Theorem~\ref{thm:limit-diff}. 
The figures in this paper were made using the Matplotlib~\cite{matplotlib} Python library.
This work was supported by NSF grants DMS-1600223 and DMS-1701576
and a Rackham Predoctoral Fellowship.

\bibliography{floor_average_ref} 
\bibliographystyle{abbrv}

\end{document}